\documentclass[11pt,leqno]{amsart}
\usepackage{amsmath,amssymb,amsthm}
\usepackage{hyperref}

\usepackage{bbm}

\renewcommand{\geq}{\geqslant}
\renewcommand{\leq}{\leqslant}

\newcommand{\norm}[1]{\left\Vert#1\right\Vert}

\newcommand{\NA}{\operatorname{NA}}

\newcommand{\supp}{\operatorname{supp}}

\newcommand{\1}[1]{\operatorname{\textbf{1}}}

\newtheorem{theorem}{Theorem}[section]
\newtheorem{lemma}[theorem]{Lemma}
\newtheorem{claim}[theorem]{Claim}

\newtheorem{corollary}[theorem]{Corollary}
\theoremstyle{definition}

\theoremstyle{remark}
\newtheorem{remark}[theorem]{Remark}
\numberwithin{equation}{section}

\def\fnote#1{\footnote}

\def\ignora#1{}
\def\n3#1{\left\vert  \! \left\vert \! \left\vert \, #1 \, \right\vert \!
  \right\vert \! \right\vert }

\newcommand{\pten}{\ensuremath{\widehat{\otimes}_\pi}}

\begin{document}

\title{ New examples of strongly subdifferentiable projective tensor products }

\author{ Abraham Rueda Zoca }\address{Universidad de Granada, Facultad de Ciencias. Departamento de An\'{a}lisis Matem\'{a}tico, 18071-Granada
(Spain)} \email{ abrahamrueda@ugr.es}
\urladdr{\url{https://arzenglish.wordpress.com}}



\maketitle

\markboth{ABRAHAM RUEDA ZOCA}{NEW EXAMPLES OF SSD PROJECTIVE TENSOR PRODUCTS}

\begin{abstract}
We prove that the norm of $X\pten Y$ is SSD if either $X=\ell_p(I)$ for $p>2$ and $Y$ is a finite-dimensional Banach space such that the modulus of convexity is of power type $q<p$ (e.g. if $Y^*$ is a subspace of $L_q$) or if $X=c_0(I)$ and $Y^*$ is any uniformly convex finite-dimensional Banach space. We also provide a characterisation of SSD elements of a projective tensor product which attain its projective norm in terms of a strengthening of the a local Bollob\'as property for bilinear mappings.
\end{abstract}

\section{Introduction}

In this paper we will study the strong subdifferentiability (SSD, for short) of projective tensor product spaces. Given a Banach space $X$, it is said that \textit{the norm of $X$ is SSD at $x\in S_X$}, or that \textit{$x$ is an SSD point}, if the one side limit
$$\lim\limits_{t\rightarrow 0^+} \frac{\Vert x+th\Vert-1}{t}$$
exists uniformly on $h\in S_X$. Observe that the norm of $X$ is Fr\'echet differentiable at $x$ if and only if it is simultaneously G\^ateaux differentiable and SSD at $x$. We say that (the norm of) $X$ is SSD if every point $x\in S_X$ is SSD. We refer the reader to \cite{aopr86,con96,fp,gmz} and references therein for background and nice characterisations of SSD. Among all the results of the above papers let us highlight \cite[Theorem 1.2]{fp} where it is proved that a point $x\in S_X$ is an SSD point if and only if given $\varepsilon>0$ there exists $\eta(\varepsilon)>0$ such that if $f\in B_{X^*}$ satisfies $f(x)>1-\varepsilon$ then there exists $g\in S_{X^*}$ such that $g(x)=1$ and $\Vert f-g\Vert<\varepsilon$. 

The above characterisation makes particularly interesting the study of SSD in projective tensor products. The reson for this interest is pointed out in \cite[Theorem 4.7]{dklm} and, in order to explain it properly, let us introduce a bit of notation from this paper. According to \cite[Definition 2.1, (a)]{dklm}, given two Banach spaces $X$ and $Y$, it is said that the pair $(X,Y)$ has \textit{the property $L_{p,p}$ for bilinear mappings} if, for every $(x,y)\in S_X\times S_Y$ and every $\varepsilon>0$, there exists $\eta(\varepsilon,x,y)>0$ such that if $B:X\times Y\longrightarrow \mathbb R$ is a norm-one bilinear mapping satisfying that $\vert B(x,y)\vert >1-\eta(\varepsilon,x,y)$ then there exists another norm-one bilinear mapping $\tilde B$ such that $\tilde B(x,y)=1$ and $\Vert B-\tilde B\Vert<\varepsilon$. This definition originally appeared motivated by a generalisation of a Bishop-Phelps-Bollob\'as property for bilinear mapping (we refer the reader to \cite[Section 1]{dklm} for more details). Now, taking into account the natural identification $(X\pten Y)^*=\mathcal B(X\times Y)$ (see Subsection~\ref{subsection:projectensor} for details) it is plain that the pair $(X,Y)$ has the property $L_{p,p}$ for bilinear mappings if and only if, the norm of $X\pten Y$ is SSD at every point of the set $\{x\otimes y: x\in S_X, y\in S_Y\}$ (it was also observed in \cite[Proposition 4.4]{djmr24}). Now, in the above mentioned \cite[Theorem 4.7]{dklm} it is proved that the pair $(\ell_p,\ell_q)$, $2<p,q<\infty$, has the $L_{p,p}$ for bilinear mappings from the fact that the dual has the $w^*$-Kadec-Klee property \cite[Theorem 4]{dk} and that this allows to conclude that the norm is SSD at every point \cite[Proposition 2.6]{dklmjmaa}. 

To the best of the author knownledge, the only known results about SSD in projective tensor products are the following ones.

\begin{enumerate}
\item If $X$ and $Y$ finite-dimensional then $X\pten Y$ is SSD at every point since it is finite-dimensional and it is clear that every point in a finite-dimensional space is SSD (see. e.g. \cite[Introduction]{gmz}).
\item $\ell_1^N\pten X$ is SSD if and only if, $X$ is SSD \cite[Theorem C]{djmr24}. This result follows because in this case $\ell_1^N\pten X=\ell_1^N(X)$ isometrically and by the characterisation of SSD norms in $\ell_1$-sums of spaces given in \cite[Proposition 2.2]{fp}.
\item $\ell_p\pten \ell_q$ is SSD if $2<p,q<\infty$ \cite[Corollary 2.8 (a)]{dklm}. This result follows since $(\ell_p\pten \ell_q)^*$ has the sequential $w^*$-Kadec-Klee property in this case \cite[Theorem 4]{dklm} and because, if $X$ is a Banach space such that $X^*$ has the sequential $w^*$-Kadec-Klee property then the norm of $X$ is SSD (see the proof of \cite[Theorem 2.7]{dklm}). This result was generalised to $\ell_p(I)\pten \ell_q(J)$ for arbitrary sets $I,J$ in \cite[Theorem 2.3]{rueda23} by making use of separable complementation arguments.
\item If $X\pten Y$ is SSD and either $X$ or $Y$ has the metric approximation property, then every bounded operator $T:X\longrightarrow Y^*$ is compact. In particular, if both $X$ and $Y$ are reflexive then $X\pten Y$ must be reflexive too \cite[Theorem 2.1]{rueda23}.
\end{enumerate}

The last point makes clear that the SSD on a projective tensor product is a restrictive property, whereas the rest of points show that the phenomenon of SSD in projective tensor products is not well understood. For instance, it is not known when $X\pten Y$ is SSD if $X$ is finite dimensional and $Y$ is reflexive.

The main aim of this paper is to provide new examples of projective tensor products whose norm is SSD and, consequently, deriving examples of new pairs of Banach spaces with the property $L_{p,p}$ for bilinear mappings. Indeed, in Theorem~\ref{theo:SSDlpyunifoconvexo} we prove that if $p>2$ and $Y$ is a finite-dimensional Banach space such that $Y^*$ has modulus of convexity of type $q$ with $q<p$, then the norm of $\ell_p(I)\pten Y$ is SSD, where $I$ is any infinite set. Furthermore, we prove in Theorem~\ref{theo:SSDc0denso} that if $I$ is an infinite set and $Y$ is a finite-dimensional smooth Banach space, then the norm of $c_0(I)\pten Y$ is SSD. In both cases, the strategy is going further and proving that the dual of the projective tensor product indeed satisfies the $w^*$-Kadec-Klee property. Finally, we provide characterisations of when an element which attains its projective norm is an SSD point in terms of a ``multiple $L_{p,p}$ condition'' on the elements of an optimal representation (see Theorems~\ref{theo:carassdpointfinito} and \ref{theo:carassdpointinfinito}).

\section{Notation and preliminary results}

For simplicity we will consider real Banach spaces. We denote by $B_X$ and $S_X$ the closed unit ball and the unit sphere, respectively, of the Banach space $X$. We denote by $\mathcal L(X, Y)$ the space of all bounded linear operators from $X$ into $Y$. If $Y = \mathbb R$, then $\mathcal L(X, \mathbb R)$ is denoted by $X^*$, the topological dual space of $X$. We also denote by $\mathcal B(X\times Y)$ the Banach space of all the bilinear and bounded mappings. 

For an abstract set $I$ we denote by $c_0(I)$ and $\ell_p(I)$ the classical generalisation of the sequence spaces $c_0$ and $\ell_p$. Given any subset $J\subseteq I$, we see $c_0(J)$ as a canonical $1$-complemented subspace of $c_0(I)$ as the space of all elements of $c_0(I)$ such that vanish on $I\setminus J$. Moreover, this canonical projection allows to consider the canonical decomposition
$$c_0(I):=\ell_\infty(J)\oplus_\infty c_0(I\setminus J).$$
In the same fashion we can consider $\ell_p(I):=\ell_p(J)\oplus_p \ell_p(I\setminus J)$ for $1\leq p<\infty$. We will make use of this identification without any further explicit mention in the rest of the text.

\subsection{Modulus of convexity}

Given a Banach space $X$, we denote the \textit{modulus of convexity $\delta_X(\varepsilon)$}, for $0<\varepsilon\leq 2$, by
\[\begin{split}
\delta_X(\varepsilon)& :=\left\{1-\frac{\Vert x+y\Vert}{2}: x,y\in X, \Vert x\Vert\leq 1, \Vert y\Vert\leq 1, \Vert x-y\Vert\geq \varepsilon \right\}
\end{split}\]
where the last inequality follows, for example, by \cite[Lemma 9.2]{checos}. $X$ is \textit{uniformly convex} if $\delta_X(\varepsilon)>0$ for every $\varepsilon>0$. 

We say that a Banach space $X$ has a modulus of convexity\textit{ of power type $q\geq 2$} if there exists $C>0$ such that $\delta_X(\varepsilon)\geq C\varepsilon^q$. An example of Banach spaces of power type of $q\geq 2$ is $L_q$. Moreover, if $1<p<2$, $L_p$ has a modulus of convexity of power type 2. See \cite[Appendix A]{beli} for more background and references. 

We also refer the reader to \cite[Chapter 9]{checos} for background on the basic connections between (uniform) convexity and (uniform) smoothness.

\subsection{Projective tensor product}\label{subsection:projectensor}

The projective tensor product of $X$ and $Y$, denoted by $X \pten Y$, is the completion of the algebraic tensor product $X \otimes Y$ endowed with the norm
$$
\|z\|_{\pi} := \inf \left\{ \sum_{n=1}^k \|x_n\| \|y_n\|: z = \sum_{n=1}^k x_n \otimes y_n \right\},$$
where the infimum is taken over all such representations of $z$. The reason for taking completion is that $X\otimes Y$ endowed with the projective norm is complete if and only if, either $X$ or $Y$ is finite dimensional (see \cite[P.43, Exercises 2.4 and 2.5]{ryan}).

It is well known that $\|x \otimes y\|_{\pi} = \|x\| \|y\|$ for every $x \in X$, $y \in Y$, and that the closed unit ball of $X \pten Y$ is the closed convex hull of the set $B_X \otimes B_Y = \{ x \otimes y: x \in B_X, y \in B_Y \}$. Throughout the paper, we will use of both facts without any explicit reference.

Observe that the action of an operator $G\colon X \longrightarrow Y^*$ as a linear functional on $X \pten Y$ is given by
$$
G \left( \sum_{n=1}^{k} x_n \otimes y_n \right) = \sum_{n=1}^{k} G(x_n)(y_n),$$
for every $\sum_{n=1}^{k} x_n \otimes y_n \in X \otimes Y$. This action establishes a linear isometry from $\mathcal L(X,Y^*)$ onto $(X\pten Y)^*$ (see e.g. \cite[Theorem 2.9]{ryan}). Observe also the natural linear isometry from $\phi:\mathcal L(X,Y^*)\longrightarrow \mathcal B(X\times Y)$ given by
$$\phi(T)(x,y):=T(x)(y)\ \ x\in X, y\in Y.$$
All along this paper we will use the isometric identifications $(X\pten Y)^*= \mathcal L(X,Y^*)=\mathcal B(X\times Y)$ without any explicit mention.

Furthermore, given two bounded operators $T\colon X\longrightarrow Z$ and $S\colon Y\longrightarrow W$, we can define an operator $T\otimes S\colon X\pten Y\longrightarrow Z\pten W$ by the action $(T\otimes S)(x\otimes y):=T(x)\otimes S(y)$ for $x\in X$ and $y\in Y$. It follows that $\Vert T\otimes S\Vert=\Vert T\Vert\Vert S\Vert$. Moreover, it is known that if $T,S$ are bounded projections then so is $T\otimes S$. Consequently, if $Z\subseteq X$ is a $1$-complemented subspace, then $Z\pten Y$ is a $1$-complemented subspace of $X\pten Y$ in the natural way (see \cite[Proposition 2.4]{ryan} for details).



\section{Main results}

Our first aim is to study the condition of SSD of $\ell_p(I)\pten Y$ for certain finite-dimensional space $Y$ and $p>2$. Let us start with the following result, which will allow us to reduce to a finite-dimensional subspace where we will be able to take advantage of the fact that finite dimensional spaces are SSD.

\begin{theorem}\label{theo:lpoperasopfin}
Let $2\leq q<p$ be numbers. Assume that $X$ has a modulus of convexity of power type $q$. Let $I$ be an infinite set. Then, for every $\varepsilon>0$ there exists $\eta>0$ satisfying that if $x\in S_{\ell_p(I)}$ is an element of finite support, say $\supp(x)= F$, and $T\in S_{\mathcal L(\ell_p(I),X)}$ satisfies that $\Vert T(x)\Vert>1-\eta$ then $\Vert T(y)\Vert<\varepsilon$ holds for every $y\in B_{\ell_p(I)}$ satisfying that $\supp(y)\subseteq I\setminus F$.
\end{theorem}

\begin{proof}
Since $X$ has modulus of convexity of power type $q$ then there exists $C>0$ satisfying $\delta_X(\varepsilon)\geq C\varepsilon^q$ holds for every $\varepsilon>0$.

Let $\varepsilon>0$. Since $p>q$ select any $q<r<p$ and choose $0<\delta<1$ small enough so that $\frac{(1+\delta^p)^\frac{1}{p}}{2 C^\frac{1}{q}}  \delta^{\frac{r}{q}-1}<\varepsilon$. Since $\delta<1$ and $r<p$ we get $\delta^p<\delta^r$. Consequently we can choose $\eta>0$ small enough to guarantee $(1-\eta)(1-\delta^p)>1-\delta^r$.

Let $T\in S_{\mathcal L(\ell_p(I),X)}$ be such that $\Vert T(x)\Vert>1-\eta$.

Set $y\in B_{\ell_p(I)}$ with $\supp(y)\subseteq I\setminus F$. By the disjointness of the supports of $x$ and $y$ we get
$$\Vert x\pm \delta y\Vert\leq (1+\delta^p)^\frac{1}{p}.$$
Hence 
$$1\geq \left\Vert \frac{T(x\pm\delta y)}{(1+\delta^p)^\frac{1}{p}} \right\Vert.$$
Define $u:=\frac{T(x+\delta y)}{(1+\delta^p)^\frac{1}{p}}$ and $v:=\frac{T(x-\delta y)}{(1+\delta^p)^\frac{1}{p}}$. Now the above inequality implies that $u,v\in B_X$ and clearly $\frac{u+v}{2}=\frac{T(x)}{(1+\delta^p)^\frac{1}{p}}$.

\begin{claim} 
$$\frac{1}{(1+\delta^p)^\frac{1}{p}}\geq 1-\delta^p.$$
\end{claim}
\begin{proof}[Proof of the claim]
Define $f:[0,\delta]\longrightarrow \mathbb R$ by $f(t):=\frac{1}{(1+t^p)^\frac{1}{p}}=(1+t^p)^\frac{-1}{p}$. Note that
$$f'(t)=-\frac{1}{p}(1+t^p)^{-\frac{1}{p}-1}pt^{p-1}=-\frac{t^{p-1}}{(1+t^p)^\frac{p+1}{p}}\ \forall t\in [0,\delta].$$
Observe that, given any $t\in [0,\delta]$ then $(1+t^p)^\frac{p+1}{p}\geq 1$ since $\frac{p+1}{p}\geq 0$. Consequently, given $t\in [0,\delta]$ we infer
$$f'(t)=-\frac{t^{p-1}}{(1+t^p)^\frac{p+1}{p}}\geq -t^{p-1}\geq -\delta^{p-1}.$$
By Mean Value Theorem there exists $\xi\in (0,\delta)$ such that 
$$f(\delta)-f(0)=f'(\xi)\delta\geq -\delta^{p-1}\delta=-\delta^p.$$
So
$$f(\delta)=\frac{1}{(1+\delta^p)^\frac{1}{p}}\geq f(0)-\delta^p=1-\delta^p.$$
\end{proof}
Now we can estimate the following norm
\[\begin{split}\left\Vert \frac{u+v}{2}\right\Vert=\frac{\Vert T(x)\Vert}{(1+\delta^p)^\frac{1}{p}}& \geq \Vert T(x)\Vert(1-\delta^p)\geq (1-\eta)(1-\delta^p)\\
& > 1-\delta^r=1-C \left(\frac{\delta^\frac{r}{q}}{C^\frac{1}{q}} \right)^q\\
& \geq 1-\delta_X\left(\frac{\delta^\frac{r}{q}}{C^\frac{1}{q}}\right).
\end{split}\]
Since $u,v\in B_X$ the unique possibility with the above inequality is that $\Vert u-v\Vert<\frac{\delta^\frac{r}{q}}{C^\frac{1}{q}}$. Now
$$u-v=\frac{2\delta T(y)}{(1+\delta^p)^\frac{1}{p}}.$$
Consequently we get
$$\Vert T(y)\Vert = (1+\delta^p)^\frac{1}{p}\frac{\Vert u-v\Vert}{2\delta}<\frac{(1+\delta^p)^\frac{1}{p}}{2 C^\frac{1}{q}}  \delta^{\frac{r}{q}-1}<\varepsilon.$$
The arbitrariness of $y\in B_{\ell_p(I)}$ with $\supp(y)\subseteq I\setminus F$ finishes the proof.
\end{proof}

The independence of $\eta$ on the finitely-supported element in the above theorem even allows to remove the restriction of taking finitely supported elements in the following sense.

\begin{corollary}\label{cor:corolpoperasopofinito}
Let $2\leq q<p$ be numbers. Assume that $X$ has a modulus of convexity of power type $q$. Let $I$ be an infinite set. Then, for every $\varepsilon>0$ there exists $\eta>0$ such that for every $x\in S_{\ell_p(I)}$ there exists a finite set $F\subseteq I$ satisfying the following: if $T\in S_{\mathcal L(\ell_p(I),X)}$ satisfies $\Vert T(x)\Vert>1-\eta$ then $\Vert T(y)\Vert<\varepsilon$ holds for every $y\in B_{\ell_p(I)}$ such that $\supp(y)\subseteq I\setminus F$.
\end{corollary}

\begin{proof}
Given $\varepsilon>0$ then we can find $\eta>0$ witnessing the property of Theorem \ref{theo:lpoperasopfin}. Now given any $x\in S_{\ell_p(I)}$, consider $x'\in S_{\ell_p(I)}$ of finite support with $\Vert x-x'\Vert<\frac{\eta}{4}$. Set $F:=\supp(x')$. Now if $T\in B_{\ell_p(I,X)}$ satisfies $\Vert T(x)\Vert>1-\frac{\eta}{4}$ then $\Vert T(x')\Vert>1-\frac{\eta}{2}>1-\eta$, and the property witnessing $\eta$ forces $\Vert T(y)\Vert<\varepsilon$ to hold for every $y\in B_{\ell_p(I)}$ with $\supp(y)\subseteq I\setminus F$.
\end{proof}

Now we are able to prove one of the main results of the paper. As we pointed out in the Introduction, the tool for obtaining the SSD in a projective tensor product will be to get that its dual space satisfies the $w^*$-Kadec-Klee property, for which we will introduce necessary notation. Following \cite[Section 1]{dk}, we say that the dual of a Banach space $X$ has the \textit{$w^*$-Kadec-Klee property} if whenever $(x_s^*)$ is a net in $S_{X^*}$ satisfying that $x_s^*\rightarrow^{w^*} x^*\in S_{X^*}$ then $\Vert x_s^*-x^*\Vert\rightarrow 0$. This is equivalent to saying that the $w^*$ topology and the norm topology agree on $S_{X^*}$. If, in the above definition, we consider sequences instead of nets, we say that $X^*$ has the \textit{sequential $w^*$-Kadec-Klee property}. See \cite{bv,hata} for background about these properties.

Our interest on the $w^*$-Kadec-Klee property comes from \cite[Proposition 2.6]{dklmjmaa}, where it is proved that if a Banach space $X$ satisfies that $X^*$ has the sequential $w^*$-Kadec-Klee property then the norm of $X$ is SSD.

Now we can establish the promised result.

\begin{theorem}\label{theo:SSDlpyunifoconvexo}
Let $p>2$, $I$ be any infinite set and let $Y$ be a finite-dimensional Banach space such that $Y^*$ has modulus of convexity of type $q$ with $q<p$. Then the norm of $\ell_p(I)\pten Y$ is SSD. Indeed, $(\ell_p(I)\pten Y)^*$ has the $w^*$-Kadec-Klee property.
\end{theorem}

\begin{proof}
Assume by contradiction that $(\ell_p(I)\pten Y)^*=\mathcal L(\ell_p(I),Y^*)$ fails the $w^*$-Kadec-Klee property. Then there exists $T\in S_{\mathcal L(\ell_p(I),Y^*)}$ and a net  $(T_s)\subseteq S_{\mathcal L(\ell_p(I),Y^*)}$ such that $(T_s)\rightarrow^{w^*} T$ but $(T_s)$ does not converge to $T$ in norm. Consequently, up to taking a subnet, we can assume that there exists $\varepsilon_0>0$ such that $\Vert T_s-T\Vert> \varepsilon_0$ holds for every $s\in S$.

By Theorem~\ref{theo:lpoperasopfin} we can find $\eta>0$ with the property that if $x$ is a finitely supported element of $B_{\ell_p(I)}$ and $G\in B_{\mathcal L(\ell_p(I),Y^*)}$ is such that $G(x)>1-\eta$ then $\Vert G(y)\Vert<\frac{\varepsilon_0}{4}$ holds for every $y\in B_{\ell_p(I)}$ with $\supp(y)\subseteq I\setminus\supp(x)$.

Select $x\in B_{\ell_p(I)}$ of finite support (say $F:=\supp(x)$) such that $\Vert T(x)\Vert>1-\eta$. Now there exists $s_1\in S$ such that $\Vert T_s(x)\Vert>1-\eta$ holds for every $s\geq s_1$. Indeed, select $y\in S_Y$ such that $T(x\otimes y)=T(x)(y)>1-\eta$. Since $T_s\rightarrow T$ in the weak star topology, in particular, $T_s(x)(y)=T_s(x\otimes y)\rightarrow T(x)(y)>1-\eta$ and, consequently, there exists $s_1\in S$ such that $s\geq s_1$ implies $1-\eta<T_s(x)(y)\leq \Vert T_s(x)\Vert$.

This implies that $s\geq s_1$ forces $\Vert T_s(y)\Vert<\frac{\varepsilon_0}{4}$ to hold for every $y\in B_{\ell_p(I)}$ such that $\supp(y)\subseteq I\setminus F$.

Now consider $S_s:=T_{s|\ell_p(F)}:\ell_p(F)\longrightarrow Y^*$ and $S:=T_{|\ell_p(F)}$. It is immediate that $(S_s)\rightarrow S$ in the $w^*$ topology of $(\ell_p(F)\pten Y)^*$ (observe that $\ell_p(F)\pten Y$ is isometrically a subspace of $\ell_p(I)\pten Y$ because $\ell_p(F)$ is a $1$-complemented subspace of $\ell_p(I)$). Since $(\ell_p(F)\pten Y)^*$ is finite dimensional we infer that $\Vert S_s-S\Vert\rightarrow 0$. Consequently we can find $s_2\in S$ such that $s_2\geq s_1$ and $s\geq s_2$ implies 
$$\frac{\varepsilon_0}{4}>\Vert S_s-S\Vert=\Vert T_{s|\ell_p(F)}-T_{|\ell_p(F)}\Vert.$$
Let us derive a contradiction from the above inequality. In order to do so, take $s\geq s_2$. Since $\Vert T_s-T\Vert\geq \varepsilon_0$ by the assumptions we can find $z\in B_{\ell_p(I)}$ such that $\varepsilon_0<\Vert T_s(z)-T(z)\Vert$. Now the decomposition $\ell_p(I)=\ell_p(F)\oplus_p \ell_p(I\setminus F)$ allows to find $z_F,z_C\in B_{\ell_p(I)}$ such that $\supp(z_F)\subseteq F, \supp(z_C)\subseteq I\setminus F, z=z_F+z_C$ and $\Vert z\Vert^p=\Vert z_F\Vert^p+\Vert z_C\Vert^p$.

Now, on the one hand, $T_s(z_F)=T_{s|\ell_p(F)}(z_F)=S_s(z_F)$. Similarly $T(z_F)=S(z_F)$, so we get
$$\Vert (T_s-T)(z_F)\Vert=\Vert (S_s-S)(z_F)\Vert\leq \Vert S_s-S\Vert \Vert z_F\Vert<\frac{\varepsilon_0}{4}\Vert z_F\Vert.$$

On the other hand, since $\supp(z_C)\subseteq I\setminus F$ we get $\Vert T_s(z_C)\Vert<\frac{\varepsilon_0}{4}$ and $\Vert T(z_C)\Vert<\frac{\varepsilon_0}{4}$.
Puting all together we get
\[
\begin{split}
\varepsilon_0<\Vert (T_s-T)(z)\Vert & =\Vert (T_s-T)(z_F+z_C)\Vert\\
& \leq \Vert (T_s-T)(z_F)\Vert+\Vert T_s(z_C)\Vert+\Vert T(z_C)\Vert\\
& <\frac{\varepsilon_0}{4}\Vert z_F\Vert+\frac{\varepsilon_0}{4}+\frac{\varepsilon_0}{4}<\frac{3\varepsilon_0}{4}<\varepsilon_0,
\end{split}
\]
a contradiction. Consequently, we deduce that $(\ell_p(I)\pten Y)^*$ has the $w^*$-Kadec-Klee property.

To conclude that the norm of $\ell_p(I)\pten Y$ is SSD it remains to apply \cite[Proposition 2.6]{dklmjmaa}.\end{proof}

Observe that given any Banach space $X$ and any subspace $Y\subseteq X$, it is immediate that $\delta_Y(\varepsilon)\geq \delta_X(\varepsilon)$. Hence, we get the following corollary.

\begin{corollary}\label{cor:aplilp(I)dualuniconv}
Let $I$ be an infinite set, $p,q$ two numbers such that $p>\max\{2,q\}$. If $Y$ is a finite-dimensional subspace of any $L_q(\mu)$ then the norm of $\ell_p(I)\pten Y^*$ is SSD.    
\end{corollary}

Roughly speaking, the key in Theorem~\ref{theo:SSDlpyunifoconvexo} is to consider spaces $\ell_p(I)$ which are ``less uniformly convex'' than the codomain space $Y^*$ (and this is crystallized by the condition of the power type on the corresponding modulus of convexity). Taking this idea beyond, in the following theorem we will replace $\ell_p(I)$ with $c_0(I)$, which even fails to be uniformly convex. In this case, however, we will not need any special behaviour of $Y^*$ else than merely being uniformly convex.

Let us begin with a version of Theorem~\ref{theo:lpoperasopfin} for $c_0(I)$. Even though the result is essentially known by speciallists (for instance see \cite[Lemma 2.11]{dklmjmaa}), we will include a complete proof for the sake of completeness and in order to get it written in the same terms as in Theorem~\ref{theo:SSDc0denso}.

\begin{lemma}\label{lemma:c0sopofinuniconvex}
Let $I$ be an infinite set and let $X$ be a uniformly convex Banach space. Then, for every $\varepsilon>0$ there exists $\eta>0$ satisfying that if $x\in S_{c_0(I)}$ is an element of finite support, say $\supp(x)= F$, and $T\in S_{\mathcal L(c_0(I),X)}$ satisfies that $\Vert T(x)\Vert>1-\eta$ then $\Vert T(y)\Vert<\varepsilon$ holds for every $y\in B_{c_0(I)}$ satisfying that $\supp(y)\subseteq I\setminus F$.    
\end{lemma}

\begin{proof}
Take $\delta_{X}:(0,2)\longrightarrow \mathbb R^+$ the modulus of convexity of $X$, which is a strictly positive function since $X$ is uniformly convex.

Let $\varepsilon>0$ and select $\eta:=\delta_{X}(\varepsilon)$, and let us prove that it satisfies our requirements. To do so, let $x\in S_{c_0(I)}$ of finite support, say $F:=\supp(x)\subseteq I$ and let $T\in S_{\mathcal L(c_0(I),X)}$ be such that $\Vert T(x)\Vert>1-\eta(\varepsilon)$.

We claim that, given any $y\in S_{c_0(I)}$ such that $\supp(y)\subseteq I\setminus F$ we have $\Vert T(y)\Vert<\varepsilon$. To this end, set $u:=T(x+y)$ and $v:=T(x-y)$. We have that $\Vert u\Vert\leq \Vert x+y\Vert\leq 1$ since $x$ and $y$ have disjoint support. Similarly $\Vert v\Vert\leq 1$. Moreover
$$\left\Vert \frac{u+v}{2}\right\Vert=\Vert T(x)\Vert>1-\delta_{X}(\varepsilon).$$
By the definition of $\delta_{X}$ the unique possibility is that $$\varepsilon>\Vert u-v\Vert= 2\Vert T(y)\Vert,$$
so $\Vert T(y)\Vert<\varepsilon$, as requested.\end{proof}

Now a repetition of the proof of Theorem~\ref{theo:SSDlpyunifoconvexo}, making use of Lemma~\ref{lemma:c0sopofinuniconvex} instead of Theorem~\ref{theo:lpoperasopfin}, allows to conclude the following theorem. We omit the proof in order to avoid repetition of details.

\begin{theorem}\label{theo:SSDc0denso}
Let $I$ be an infinite set and let $Y$ be a finite-dimensional smooth Banach space. Then the norm of $c_0(I)\pten Y$ is SSD. Indeed, $(c_0(I)\pten Y)^*$ has the $w^*$-Kadec-Klee property. 
\end{theorem}

\begin{remark}\label{remark:lppexamplesnew}
Observe that Theorems \ref{theo:SSDlpyunifoconvexo} and \ref{theo:SSDc0denso} yield examples of pairs of Banach spaces $X$ and $Y$ such that the pair $(X,Y)$ has the property $L_{p,p}$ for bilinear mappings in the case that $Y$ is finite dimensional which, up to our knownledge, were not previously known.

In the context of finite-dimensional Banach spaces, it is known that if both $X$ and $Y$ are finite-dimensional then the pair $(X,Y)$ has the $L_{p,p}$ for bilinear mappings \cite[Proposition 2.2]{dklm}. However, as far as we there is not any good description of when $(X,Y)$ has the $L_{p,p}$ for bilinear mappings if $X$ is finite-dimensional, with the exception of $X=\ell_1^n$ \cite[Theorem 3.3]{dklm}.
\end{remark}

Let $X$ and $Y$ be two Banach spaces. It is asked in \cite[Subsection 2.4, Q1)]{djmr24} whether $X\pten Y$ is SSD if and only if, the pair $(X,Y)$ has the $L_{p,p}$ for bilinear mappings. We do not know whether the answer is affirmative or negative. However, let us point out that, in order to get an SSD projective tensor product, we will need a strengthening of the $L_{p,p}$ condition to multiple points. 

In order to motivate this, consider an arbitrary Banach space $Z$ which is SSD. Assume that $z_1,\ldots, z_n\in S_Z$ are points such that there exists a common functional $f$ such that $f(z_i)=1$ holds for every $1\leq i\leq n$ (this condition is clearly equivalent by Hahn-Banach theorem to the fact that $\Vert z_1+\ldots+z_n\Vert=n$). Now set $z:=\frac{1}{n}\sum_{i=1}^n z_i\in S_Z$. Since $z$ is an SSD point then, given $\varepsilon>0$ there exists $\eta(\varepsilon)>0$ (which of course also depends on the point $z$) with the property that if $g\in B_{Z^*}$ is such that $g(z)>1-\eta(\varepsilon)$ then there exists $h\in S_{Z^*}$ such that $h(z)=1$ and $\Vert g-h\Vert<\varepsilon$.

This implies a multiple point SSD condition in the following sense: if $g\in B_{Z^*}$ is such that $g(z_i)>1-\eta(\varepsilon)$ holds for every $1\leq i\leq n$, then $g(z)>1-\eta(\varepsilon)$ by a convexity argument, so there exists $h\in S_{Z^*}$ with $h(z)=1=\frac{1}{n}\sum_{i=1}^n h(z_i)$ and $\Vert g-h\Vert<\varepsilon$. The condition $h(z)=1$ implies $h(z_i)=1$ to hold for every $1\leq i\leq n$. 

Summarising, we have obtained the following necessary condition: if a Banach space $Z$ is SSD then, given $z_1,\ldots, z_n\in S_Z$  such that there exists a functional $f\in S_{Z^*}$ such that $f(z_i)=1$ holds for every $1\leq i\leq n$ then, for every $\varepsilon>0$, there exists $\eta(\varepsilon)>0$ with the property that if $g\in B_{Z^*}$ is such that $g(z_i)>1-\eta(\varepsilon)$ holds for every $1\leq i\leq n$ then there exists $h\in S_{Z^*}$ such that $h(z_i)=1$ and $\Vert g-h\Vert<\varepsilon$.

Moving this to the context of the projective tensor product we can obtain a characterisation of the property that certain elements of the projective tensor product are SSD.

\begin{theorem}\label{theo:carassdpointfinito}
Let $X$ and $Y$ be two Banach space, let $n\in\mathbb N$ and $x_1,\ldots, x_n\in S_X$, $y_1,\ldots, y_n\in S_Y$ such that $\frac{1}{n}\left\Vert \sum_{i=1}^n x_i\otimes y_i\right\Vert=1$ (this condition is equivalent to the existence of a norm-one bilinear mapping $V$ such that $V(x_i,y_i)=1$ holds for every $1\leq i\leq n$). The following are equivalent:
\begin{enumerate}
\item For every $\lambda_1,\ldots,\lambda_n\in (0,1]$ such that $\sum_{i=1}^n\lambda_i=1$, the element $\sum_{i=1}^n \lambda_i x_i\otimes y_i$ is an SSD point of $S_{X\pten Y}$.
\item For every $\varepsilon>0$ there exists $\eta(\varepsilon)>0$ (also depending on the points $x_1,\ldots x_n, y_1,\ldots, y_n$) satisfying the following condition: if $B\in \mathcal B(X\times Y)$ satisfies $\Vert B\Vert\leq 1$ and $B(x_i,y_i)>1-\eta(\varepsilon)$ holds for every $1\leq i\leq n$ then there exists $\tilde B\in S_{\mathcal B(X\times Y)}$ satisfying that $\tilde B(x_i,y_i)=1$ holds for every $1\leq i\leq n$ and $\Vert B-\tilde B\Vert<\varepsilon$.
\end{enumerate}
\end{theorem}

\begin{proof}
(1)$\Rightarrow$(2). Let $\varepsilon>0$. Let us find $\eta(\varepsilon)>0$ enjoying the condition that if $B\in B_{\mathcal B(X\times Y)}$ is such that $B(x_i,y_i)>1-\eta(\varepsilon)$ holds for $1\leq i\leq n$ then we can find $\tilde B\in S_{\mathcal B(X\times Y)}$ with $\tilde B(x_i,y_i)=1$ holds for $1\leq i\leq n$ and $\Vert B-\tilde B\Vert<\varepsilon$.

Since $z:=\frac{1}{n}\sum_{i=1}^n x_i\otimes y_i$ is an SSD point of $S_{X\pten Y}$ by (1) we can find $\eta(\varepsilon)>0$ such that if $B\in B_{\mathcal B(X\times Y)}$ satisfies that 
$$1-\eta(\varepsilon)<B\left( z\right)=\frac{1}{n}\sum_{i=1}^n B(x_i,y_i)$$
then there exists $\tilde B\in S_{\mathcal B(X\times Y)}$ with $\tilde B(z)=1$ and $\Vert B-\tilde B\Vert<\varepsilon$.

Let us prove that $\eta(\varepsilon)$ satisfies our purposes. Indeed, if $B\in B_{\mathcal B(X\times Y)}$ satisfies that $B(x_i,y_i)>1-\eta(\varepsilon)$ for every $1\leq i\leq n$ then 
$$B(z)=\frac{1}{n}\sum_{i=1}^n B(x_i,y_i)>\frac{1}{n}\sum_{i=1}^n (1-\eta(\varepsilon))=1-\eta(\varepsilon).$$
By the property defining $\eta(\varepsilon)$ we can find $\tilde B\in S_{\mathcal B(X\times Y)}$ with $\tilde B(z)=1$ and $\Vert B-\tilde B\Vert<\varepsilon$. It remains to prove that $\tilde B(x_i,y_i)=1$. But this holds true since
$$1=\tilde B(z)=\frac{1}{n}\sum_{i=1}^n \tilde B(x_i,y_i),$$
and since $\tilde B(x_i,y_i)\leq \Vert \tilde B\Vert=1$, an easy convexity argument together with the above equality implies $\tilde B(x_i,y_i)=1$ holds for every $1\leq i\leq n$, as desired.

(2)$\Rightarrow$(1). Fix $\lambda_1,\ldots, \lambda_n\in (0,1]$ such that $\sum_{i=1}^n \lambda_i=1$ and $z:=\sum_{i=1}^n \lambda_i x_i\otimes y_i$, and let us prove that it is an SSD point. To this end, let $\varepsilon>0$ and let us find $\delta(\varepsilon)>0$ such that if $B\in B_{\mathcal B(X\times Y)}$ satisfies $B(z)>1-\delta(\varepsilon)$ then we can find $\tilde B\in S_{\mathcal B(X\times Y)}$ with $\tilde B(z)=1$ and $\Vert B-\tilde B\Vert<\varepsilon$. 

By (2) there exists $\eta(\varepsilon)>0$ (also depending on the points $x_1,\ldots x_n, y_1,\ldots, y_n$) enjoying the condition that if $B\in B_{\mathcal B(X\times Y)}$ is such that $B(x_i,y_i)>1-\eta(\varepsilon)$ holds for $1\leq i\leq n$ then we can find $\tilde B\in S_{\mathcal B(X\times Y)}$ with $\tilde B(x_i,y_i)=1$ holds for $1\leq i\leq n$ and $\Vert B-\tilde B\Vert<\varepsilon$.

Let $\lambda:=\min_{1\leq i\leq n}\lambda_i>0$ and define $\delta(\varepsilon):=\lambda\eta(\varepsilon)$, which is strictly positive. Let us prove that $\delta(\varepsilon)$ satisfies our requirements. 

To this end, let $B\in B_{\mathcal B(X\times Y)}$ such that $B(z)>1-\delta(\varepsilon)$. We claim that $B(x_i,y_i)>1-\eta(\varepsilon)$ holds for every $1\leq i\leq n$. Indeed, assume by contradiction that $B(x_i,y_i)\leq 1-\eta(\varepsilon)$ for some $1\leq i\leq n$. Then
\[\begin{split}
1-\lambda\eta(\varepsilon)<B(z)=\sum_{j=1}^n \lambda_j B(x_j,y_j)& =\lambda_i B(x_i,y_i)+\sum_{j\neq i}\lambda_j B(x_j,y_i)\\
& \leq \lambda_i (1-\eta(\varepsilon))+\sum_{j\neq i}\lambda_j\\
& =\lambda_i(1-\eta(\varepsilon))+1-\lambda_i\\
& =1-\lambda_i\eta(\varepsilon).\end{split}\]
The above inequality implies $\lambda_i \eta(\varepsilon)<\lambda \eta(\varepsilon)$ which implies $\lambda_i<\lambda=\min\{\lambda_1,\ldots, \lambda_n\}$, a contradiction. This proves $B(x_i,y_i)>1-\eta(\varepsilon)$ for every $1\leq i\leq n$. The condition defining $\eta(\varepsilon)$ yields the existence of $\tilde B\in S_{\mathcal B(X\times Y)}$ with $\tilde B(x_i,y_i)=1$ for every $1\leq i\leq n$ and $\Vert B-\tilde B\Vert<\varepsilon$. It remains to prove that $\tilde B(z)=1$. But this is clear because
$$\tilde B(z)=\tilde B\left(\sum_{i=1}^n \lambda_i x_i\otimes y_i \right)=\sum_{i=1}^n \lambda_i \tilde B(x_i\otimes y_i)=\sum_{i=1}^n \lambda_i B(x_i,y_i)=\sum_{i=1}^n \lambda_i =1,$$
which concludes (1).
\end{proof}

\begin{remark}
In view of the above theorem, if a projective tensor product $X\pten Y$ is SSD, then the pair $(X,Y)$ satisfies the following: for every $(x_1,y_1), \ldots, (x_n,y_n)\in S_X\times S_Y$ such that there exists a norm-one bilinear map $G\in \mathcal B(X\times Y)$ such that $G(x_i,y_i)=1$ holds for every $1\leq i\leq n$ the following condition holds: for every $\varepsilon>0$ there exists $\eta(\varepsilon)>0$ such that if a bilinear mapping $B\in B_{\mathcal B(X\times Y)}$ satisfies $B(x_i,y_i)>1-\eta(\varepsilon)$ then there exists $\tilde B\in S_{\mathcal B(X\times Y)}$ such that $B(x_i,y_i)=1$ holds for every $1\leq i\leq n$ and $\Vert B-\tilde B\Vert<\varepsilon$. 

On the other hand obverve that, in virtue of Theorems~\ref{theo:SSDlpyunifoconvexo} and \ref{theo:SSDc0denso}, the pair $(X,Y)$ satisfies the above property if $X=\ell_p(I)$ for $p>2$ and $Y$ is finite dimensional such that $Y^*$ has modulus of convexity of power type $q<2$ or if $X=c_0(I)$ and $Y$ is any finite dimensional smooth Banach space.
\end{remark}

Let us finish the paper with a transfinite version of Theorem~\ref{theo:carassdpointfinito}.

\begin{theorem}\label{theo:carassdpointinfinito}
Let $X$ and $Y$ be two Banach spaces and set $\{x_n\}\subseteq S_X$ and $\{y_n\}\subseteq S_Y$ such that $\left\Vert \sum_{n=1}^\infty \frac{1}{2^n}x_n\otimes y_n\right\Vert=1$ (this condition is equivalent to the existence of a norm-one bilinear mapping $V$ such that $V(x_n,y_n)=1$ holds for every $n\in\mathbb N$). The following are equivalent:
\begin{enumerate}
\item $\sum_{n=1}^\infty \lambda_n x_n\otimes y_n$ is an SSD point of $S_{X\pten Y}$ for every $\lambda_n\in (0,1]$ such that $\sum_{n=1}^\infty \lambda_n=1$.
\item For every $\varepsilon>0$ there exist $k_\varepsilon\in\mathbb N$ and $\eta(\varepsilon)>0$ such that if $B\in B_{\mathcal B(X\times Y)}$ satisfies $B(x_n,y_n)>1-\delta(\varepsilon)$ holds for every $1\leq n\leq k_\varepsilon$ then there exists $\tilde B\in S_{\mathcal B(X\times Y)}$ such that $\Vert B-\tilde B\Vert<\varepsilon$ and $B(x_n,y_n)=1$ holds for every $n\in\mathbb N$.
\end{enumerate}
\end{theorem}

\begin{proof}
(1)$\Rightarrow$(2). Let $\varepsilon>0$ and define $z:=\sum_{n=1}^\infty \frac{1}{2^n}x_n\otimes y_n$. Since $z$ is an SSD point of $S_{X\pten Y}$ there exists $\delta(\varepsilon)>0$ such that if $B\in B_{\mathcal B(X\times Y)}$ satisfies $B(z)>1-\delta(\varepsilon)$ then there exists $\tilde B\in S_{\mathcal B(X\times Y)}$ with $\tilde B(z)=1$ and $\Vert B-\tilde B\Vert<\varepsilon$.

Select $k_\varepsilon\in\mathbb N$ big enough to guarantee $\sum_{n=1}^{k_\varepsilon}\frac{1}{2^n}-\sum_{n=k_\varepsilon+1}^\infty\frac{1}{2^n}>1-\delta(\varepsilon)$. The strict inequality allows to find $\eta(\varepsilon)$ small enoguh to get
$$(1-\eta(\varepsilon))\sum_{n=1}^{k_\varepsilon}\frac{1}{2^n}-\sum_{n=k_\varepsilon+1}^\infty \frac{1}{2^n}>1-\delta(\varepsilon).$$
Let us prove that $k_\varepsilon$ and  $\eta(\varepsilon)$ satisfies our requirements. In order to do so select $B\in B_{\mathcal B(X\times Y)}$ such that $B(x_n,y_n)>1-\eta(\varepsilon)$ holds for every $1\leq n\leq k_\varepsilon$. Then
\[\begin{split}
B(z)& =\sum_{n=1}^{k_\varepsilon} \frac{1}{2^n} B(x_n,y_n)+\sum_{n=k_\varepsilon+1}^\infty \frac{1}{2^n}B(x_n,y_n)\\
& >(1-\eta(\varepsilon))\sum_{n=1}^{k_\varepsilon}\frac{1}{2^n}-\sum_{n=k_\varepsilon+1}^\infty \frac{1}{2^n}>1-\delta(\varepsilon).
\end{split}\]
The property defining $\delta(\varepsilon)$ implies the existence of $\tilde B\in S_{\mathcal B(X\times Y)}$ such that $\tilde B(z)=1$ and $\Vert B-\tilde B\Vert<\varepsilon$. Since $1=B(z)=\sum_{n=1}^\infty \frac{1}{2^n}B(x_n,y_n)$ and $\sum_{n=1}^\infty \frac{1}{2^n}=1$, a convexity argument implies $B(x_n,y_n)=1$ holds for every $n\in\mathbb N$, concluding (2).

To prove (2)$\Rightarrow$(1), choose $\lambda_n\in (0,1]$ for every $n\in\mathbb N$ such that $\sum_{n=1}^\infty \lambda_n=1$, write $z=\sum_{n=1}^\infty \lambda_n x_n\otimes y_n$ and let us prove that $z$ is an SSD point of $S_{X\pten Y}$. In order to do so, let $\varepsilon>0$ and let us find $\eta(\varepsilon)>0$ with the property that if $B\in B_{\mathcal B(X\times Y)}$ is such that $B(z)>1-\eta(\varepsilon)$ then there exists $\tilde B\in S_{\mathcal B(X\times Y)}$ with $\tilde B(z)=1$ and $\Vert B-\tilde B\Vert<\varepsilon$. 

By condition (2) there exists $k_\varepsilon\in\mathbb N$ and $\delta(\varepsilon)>0$ with the property that if $B\in B_{\mathcal B(X\times Y)}$ satisfies $B(x_n,y_n)>1-\delta(\varepsilon)$ for every $1\leq n\leq k_\varepsilon$ then there exists $\tilde B\in S_{\mathcal B(X\times Y)}$ such that $\Vert B-\tilde B\Vert<\varepsilon$ and $\tilde B(x_n,y_n)=1$ holds for every $n\in\mathbb N$.

Set $\lambda:=\min_{1\leq i\leq k_\varepsilon}\lambda_i>0$ and take $\eta(\varepsilon):=\lambda\delta(\varepsilon)$. Let us prove that $\eta(\varepsilon)$ satisfies the requirements. In order to do so, let $B\in B_{\mathcal B(X\times Y)}$ such that $B(z)>1-\eta(\varepsilon)$. Let us prove that $B(x_n,y_n)>1-\delta(\varepsilon)$ holds for every $1\leq n\leq k_\varepsilon$. To do so define
$$P:=\{n\in\mathbb N: B(x_n,y_n)\leq 1-\delta(\varepsilon)\}.$$
We have
\[
\begin{split}
1-\eta(\varepsilon)=1-\lambda\delta(\varepsilon)& <B(z)=\sum_{n=1}^\infty \lambda_n B(x_n,y_n)\\
& =\sum_{n\in P}\lambda_n B(x_n,y_n)+\sum_{n\notin P}\lambda_n B(x_n,y_n)\\
& \leq (1-\delta(\varepsilon))\sum_{n\in P}\lambda_n +\sum_{n\notin P}\lambda_n\\
& =(1-\delta(\varepsilon))\sum_{n\in P}\lambda_n +1-\sum_{n\in P}\lambda_n\\
& =1-\delta(\varepsilon)\sum_{n\in P}\lambda_n.
\end{split}
\]
This implies that $\sum_{n\in P}\lambda_n<\lambda=\min\{\lambda_1,\ldots, \lambda_{k_\varepsilon}\}$. From this inequality we infer that $\{1,\ldots, k_\varepsilon\}\cap P=\emptyset$, in other words, $B(x_n,y_n)>1-\delta(\varepsilon)$ holds for every $1\leq n\leq k_\varepsilon$. By the property defining $\delta(\varepsilon)$ there exists $\tilde B\in S_{\mathcal B(X\times Y)}$ such that $\tilde B(x_n,y_n)=1$ holds for every $n\in\mathbb N$ and $\Vert B-\tilde B\Vert<\varepsilon$. This finishes the proof since $\tilde B(z)=1$. Indeed
$$\tilde B(z)=\sum_{n=1}^\infty \lambda_n \tilde B(x_n,y_n)=\sum_{n=1}^\infty \lambda_n=1,$$
and the proof is finished.
\end{proof}

\begin{remark}
According to \cite[Definition 2.1]{djrr22}, an element $z\in X\pten Y$ is said to \textit{attain its projective norm} if there exists a sequence $(x_n)_n$ in $X$ and $(y_n)_n$ in $Y$ such that $\Vert z\Vert=\sum_{n=1}^\infty \Vert x_n\Vert \Vert  y_n\Vert$ and $z=\sum_{n=1}^\infty x_n\otimes y_n$. We denote $\NA_\pi(X\pten Y)$ the set of those $z$ which attain their projective norm.
Observe that for all $z\in \NA_\pi(X\pten Y)$ we can write $z=\sum_{n=1}^\infty \lambda_n x_n\otimes y_n$, where $x_n\in S_X, y_n\in S_Y$ for every $n\in\mathbb N$ and $\lambda_n \in \mathbb R^+$ satisfy $\sum_{n=1}^\infty \lambda_n=\Vert z\Vert$. In such case, given any $T\in  \mathcal L(X,Y^*)=(X\pten Y)^*$ such that $T(z)=\Vert z\Vert$ and $\Vert T\Vert=1$, a convexity argument implies that $T(x_n)(y_n)=\norm{T(x_n)}=\norm{T}=1$ for every $n\in\mathbb N$. We refer the reader to \cite{dgljrz23,djmr24,djrr22,ggr24,rueda23} for background on projective norm-attainment.

Observe that Theorem~\ref{theo:carassdpointinfinito} could be a potential tool in order to get examples of projective tensor products $X\pten Y$ for which the set of SSD points is dense.  Indeed, it turns out to be a characterisation of when a norm-attaining element $z\in \NA_\pi(X\pten Y)$ is SSD in terms of a $L_{p,p}$ condition involving the pairs of points of an optimal representation. On the other hand, the phenomenon that $\NA_\pi(X\pten Y)$ is dense in $X\pten Y$ is quite frequent. As a matter of fact, in \cite[Theorem 4.8]{djrr22} it is proved that if $X$ and $Y$ have the metric $\pi$-property then $\NA_\pi(X\pten Y)$ is dense in $X\pten Y$.
\end{remark}

\section*{Acknowledgements}  

The author is deeply grateful to Sheldon Dantas for a number of comments, questions and remarks that have sensibly improved the exposition.

This work was supported by MCIN/AEI/10.13039/501100011033: grant PID2021-122126NB-C31, Junta de Andaluc\'ia: grant FQM-0185, by Fundaci\'on S\'eneca: ACyT Regi\'on de Murcia: grant 21955/PI/22 and by Generalitat Valenciana: grant CIGE/2022/97.

\end{document}